\documentclass[12pt]{amsart} 
\usepackage{amssymb,latexsym, amscd, wasysym, stmaryrd, cite} 
\usepackage[mathscr]{eucal}
\theoremstyle{plain}
\newtheorem{thm}{Theorem}
\newtheorem{cor}{Corollary}
\newtheorem{lem}{Lemma}
\newtheorem{prop}{Proposition}

\newtheorem*{namedtheorem}{\theoremname}
\newcommand{\theoremname}{testing}
\newenvironment{named}[1]{\renewcommand{\theoremname}{#1}
    \begin{namedtheorem}}
    {\end{namedtheorem}}

\theoremstyle{definition}

\newtheorem{ex}{Example}
\newtheorem{remark}{Remark}

\numberwithin{equation}{section}

\renewcommand{\leq}{\leqslant}
\renewcommand{\geq}{\geqslant}
\newcommand{\norm}{\trianglelefteqslant}
\newcommand{\cin}{\subseteq}

\newcommand{\nin}{\notin}

\newcommand{\setst}[2]{\{\,#1\mid #2 \}}

\newcommand{\inv}[1]{#1^{-1}}

\newcommand{\calF}{\mathcal{F}}
\newcommand{\calC}{\mathcal{C}}

\newcommand{\id}{\operatorname{id}}

\newcommand{\Hom}{\operatorname{Hom}}

\newcommand{\SL}{\operatorname{SL}}

\newcommand{\PSL}{\operatorname{PSL}}
\newcommand{\Aut}{\operatorname{Aut}}
\newcommand{\Out}{\operatorname{Out}}

\newcommand{\Alt}{\operatorname{Alt}}

\begin{document}

\title[Goldschmidt's thesis for fusion systems]{Analogues of Goldschmidt's thesis for fusion systems}
\author{Justin Lynd}
\address{Department of Mathematics, The Ohio State University, Columbus, OH 43210}
\email{jlynd@math.ohio-state.edu}
\author{Sejong Park}
\address{Department of Mathematical Sciences, University of Aberdeen, Aberdeen, UK  AB24 3UE}
\email{s.park@abdn.ac.uk}
\date{August 24, 2010}

\begin{abstract}
We extend the results of David Goldschmidt's thesis concerning fusion in finite groups 
to saturated fusion systems and to all primes.
\end{abstract}
\maketitle

\section{Introduction}

Recently, David Goldschmidt published his doctoral thesis \cite{Goldschmidt2008} which had gone unpublished since
1968. In it he shows that if $G$ is a finite simple group and $T$ is a Sylow $2$-subgroup of $G$, then
the exponent of $Z(T)$ (and hence of $T$) is bounded by a function of the nilpotence class
of $T$.
He also includes in the write-up a fusion factorization result for an arbitrary finite group
involving $\mho^1Z$ and the Thompson subgroup. 
In this paper, we generalize these results to saturated fusion systems.

Throughout this paper unless otherwise indicated, $p$ will be a prime number, $n$ a nonnegative integer,
and $P$ a nontrivial finite $p$-group.

\begin{thm}\label{main}
Suppose $P$ is of nilpotence class at most 
$n(p-1) + 1$, and $\calF$ is a saturated fusion system on $P$ with $O_p(\calF) = 1$. Then $Z(P)$ has exponent
at most $p^n$. 
\end{thm}
This bound is sharp for all $n$ and $p$; see Example \ref{exsharp} in Section 3. This also gives a bound on
the exponent of $P$ itself, which we certainly do not expect to be sharp.
\begin{cor}
Suppose that $P$ is of nilpotence class at most $n(p-1)+1$, and $\calF$ is a saturated fusion system on $P$ with
$O_p(\calF) = 1$. Then $P$ has exponent at most $p^{n^2(p-1) + n}$. 
\end{cor}
\begin{proof}
By Theorem \ref{main}, $Z(P)$ has exponent at most $p^n$. We claim that then every upper central
quotient also has exponent at most $p^n$, and we shall prove this by induction. Let $k \geq 1$, and let
$x \in Z^{k+1}(P)$. If $x^{p^n}$ does not lie in $Z^k(P)$, then there exists $t \in P$ such that
$[x^{p^n}, t]$ does not lie in $Z^{k-1}(P)$. But by a standard commutator identity, $
[x^{p^n}, t] \equiv [x, t]^{p^n} \equiv 1$ modulo $Z^{k-1}(P)$, since by induction $Z^{k}(P)/Z^{k-1}(P)$
has exponent at most $p^n$. This contradiction establishes the claim. The nilpotence class of $P$ is
at most $n(p-1)+1$ by hypothesis, so the exponent of $P$ is at most $p^{n(n(p-1)+1)}$.
\end{proof}

Theorem \ref{main} is obtained from the following.

\begin{thm}\label{norm}
Suppose $P$ has nilpotence class at most $n(p-1)+1$ and $\calF$ is a saturated fusion system on $P$. Then $\mho^n(Z(P))$
is normal in $\calF$. 
\end{thm}

In the course of proving this last result 
in the group case for $p=2$, Goldschmidt reduces to the 
situation in which a putative counterexample $G$ has a weakly embedded $2$-local subgroup. 
Then his post-thesis classification \cite{Goldschmidt1972} of such groups 
gives a contradiction. However, any weakly embedded $p$-local $M$ controls $p$-fusion, and so
the $p$-subgroup $O_p(M)$ will show up as a normal subgroup in the fusion system, a shadow of the weakly
embedded phenomenon. This allows the corresponding fusion result to hold for an arbitrary prime.

We note that Theorem \ref{norm} has the following corollary in the category of groups.

\begin{thm}\label{maingrp}
Let $P$ be a nonabelian Sylow $p$-subgroup of a finite group $G$. Suppose that $P$ has nilpotence
class at most $n(p-1) + 1$ and that $G$ has no nontrivial strongly closed abelian $p$-subgroup.
Then $Z(P)$ has exponent at most $p^n$. 
\end{thm}
\begin{proof}
We can form the saturated fusion system $\calF_P(G)$, and Theorem \ref{norm} then says that
$\mho^n(Z(P))$ is strongly $\calF$-closed (see Proposition \ref{equivnorm} below), 
that is, strongly closed in $P$ with respect to $G$.
Thus, $\mho^n(Z(P))$ must be trivial.
\end{proof}

Using a recent theorem of Flores and Foote \cite{FloresFoote2009}, in which
they apply the Classification of Finite Simple Groups to describe all finite
groups having a strongly closed $p$-subgroup, we get the following direct
generalization of Goldschmidt's main theorem.  Note that Corollary \ref{corgrp}
is the only result in this paper which relies on a deep classification theorem.

\begin{cor}\label{corgrp}
Let $P$ be a nonabelian Sylow $p$-subgroup of a finite simple group $G$. If $P$ has nilpotence
class at most $n(p-1)+1$, then $Z(P)$ has exponent at most $p^n$. 
\end{cor}
\begin{proof}
Suppose to the contrary that $A := \mho^n(Z(P)) \neq 1$. Then by Theorem \ref{norm}, $A$
is a nontrivial strongly closed abelian subgroup of $P$. By inspection of the simple groups
arising in the conclusion of the main theorem in \cite{FloresFoote2009}, either $P$ is
abelian or $Z(P)$ has exponent $p$. Since $P$ is nonabelian, we must have $n \geq 1$ and
the corollary follows.
\end{proof}

However, if the hypotheses of Corollary \ref{corgrp} are weakened slightly to assume only that 
$F^{*}(G)$ is simple, then the
statement is false for all odd primes $p$, as the following example shows. Let $H = \PSL(2,q)$
with $q = r^p$ for some prime power $r$ and with the $p$-part of $q-1$ equal to $p^e$.
Let $\sigma$ be a field automorphism of ${\bf F}_q$ of order $p$ and 
$G = H\langle \sigma \rangle$. If $P$ is a Sylow $p$-subgroup of $G$, then $P$ has
nilpotence class $2$, while $Z(P)$ has exponent $p^{e-1}$, and we may take $e$ as large as we like. 

Recall the Thompson subgroup $J(P)$, defined as the group generated by the abelian subgroups of $P$
of maximum order. We also prove the following factorization result.

\begin{thm}\label{fact}
Let $\calF$ be a saturated fusion system on $P$. Then
\[
\calF = \langle\,\,C_\calF(\mho^1(Z(P)),\, N_{\calF}(J(P))\,\,\rangle.
\]
\end{thm}

\section{Definitions and notation}

We collect in this section the necessary information on fusion systems. Since there are by now many good
sources of this knowledge \cite{BrotoLeviOliver2003}, in particular in background sections of papers 
\cite{DiazGlesserMazzaPark2009, KessarLinckelmann2008} to which this one is similar, we will content ourselves to be brief.

Let $P$ be a finite $p$-group. A \emph{category on $P$} is a category $\calF$ with objects the subgroups
of $P$ and whose morphism sets $\Hom_\calF(Q, R)$ consist of injective group homomorphisms subject to
the requirement that every morphism in $\calF$ is a composition of an isomorphism in $\calF$ and an inclusion.

Let $\calF$ be a category on the $p$-group $P$. Let $Q$ and $R$ be subgroups of $P$. We write $\Aut_\calF(Q)$ for
$\Hom_{\calF}(Q, Q)$, $\Hom_P(Q, R)$ for the set of group homomorphisms in $\calF$ from $Q$ to $R$ induced by 
conjugation by elements of $P$, and $\Out_\calF(Q)$ for $\Aut_\calF(Q)/\Aut_Q(Q)$. 

We say $Q$ is
\begin{itemize}
\item \emph{fully $\calF$-normalized} if $|N_P(Q)| \geq |N_P(Q')|$ for all $Q'$ which are $\calF$-isomorphic to $Q$,
\item \emph{fully $\calF$-centralized} if $|C_P(Q)| \geq |C_P(Q')|$ for all $Q'$ which are $\calF$-isomorphic to $Q$,
\item \emph{$\calF$-centric} if $C_P(Q') \leq Q'$ for all $Q'$ which are $\calF$-isomorphic to $Q$, and
\item \emph{$\calF$-radical} if $O_p(\Out_\calF(Q)) = 1$. 
\end{itemize}

For a morphism $\varphi: Q \to P$ in $\calF$, let
\[
N_\varphi = \setst{x \in N_P(Q)}{\exists y \in N_P(\varphi(Q)),\, \forall z \in Q,\, \varphi(xz\inv{x}) = y\varphi(z)\inv{y}}
\]
Note that we have $QC_P(Q) \leq N_\varphi$ for all $\varphi: Q \to P$ in $\calF$.

A \emph{saturated fusion system} on $P$ is a category $\calF$ on $P$ whose morphism sets 
contain all group homomorphisms induced by conjugation by elements of $P$, and which satisfies the 
following two axioms. 
\begin{itemize}
\item (Sylow axiom) $\Aut_P(P)$ is a Sylow $p$-subgroup of $\Aut_\calF(P)$, and
\item (Extension axiom) for every isomorphism $\varphi: Q \to Q'$ with $Q'$ fully $\calF$-normalized, there exists
a morphism $\tilde{\varphi}: N_\varphi \to P$ such that $\tilde{\varphi}\,|_Q = \varphi$. 
\end{itemize}

For the remainder of the paper, $\calF$ will denote a saturated fusion system on the finite $p$-group $P$,
even though we will often drop the adjective ``saturated''. 

For $Q \leq P$, we define the following local subcategories of $\calF$. The \emph{normalizer} $N_\calF(Q)$
of $Q$ in $\calF$ is the category on $N_P(Q)$ such that for any $R_1, R_2 \leq N_P(Q)$, 
$\Hom_{N_\calF(Q)}(R_1, R_2)$ consists of those $\varphi: R_1 \to R_2$ in $\calF$ for which there is
an extension $\tilde{\varphi}: QR_1 \to QR_2$ of $\varphi$ in $\calF$ such that $\tilde{\varphi}(Q) = Q$.
The \emph{centralizer} $C_\calF(Q)$ of $Q$ in $\calF$ is the category on $C_P(Q)$ such that for any 
$R_1, R_2 \leq C_P(Q)$, 
$\Hom_{C_\calF(Q)}(R_1, R_2)$ consists of those $\varphi: R_1 \to R_2$ in $\calF$ for which there is
an extension $\tilde{\varphi}: QR_1 \to QR_2$ of $\varphi$ in $\calF$ such that $\tilde{\varphi}\,|_Q = \id_Q$.
Lastly, we define $N_P(Q)C_\calF(Q)$ as we do the normalizer of $Q$, but only allow those $\varphi: R_1 \to R_2$ whose 
extensions $\tilde{\varphi}$ restrict to automorphisms in $\Aut_P(Q)$.

If $Q$ is fully $\calF$-normalized, then $N_\calF(Q)$ is a saturated fusion system. And if $Q$ is fully
$\calF$-centralized, then both $C_\calF(Q)$ and $N_P(Q)C_\calF(Q)$ are saturated fusion systems.

A \emph{characteristic functor} is a mapping from finite $p$-groups to finite $p$-groups which takes $Q$ to a characteristic
subgroup $W(Q)$ of $Q$ such that for any group isomorphism $\varphi: Q \to Q'$, $\varphi(W(Q)) = W(Q')$. We say
that a characteristic functor is \emph{positive} provided $W(Q) \neq 1 $ whenever $Q \neq 1$. The \emph{center
functor}, sending a finite $p$-group $P$ to its center, is a positive characteristic $p$-functor.

A \emph{conjugation family} for $\calF$ is a set $\calC$ of nonidentity subgroups of $P$ such that $\calF$ is
generated by compositions and restrictions of morphisms in $\Aut_\calF(Q)$ as $Q$ ranges over $\calC$. Alperin's
fusion theorem for saturated fusion systems says that the set of $\calF$-centric, $\calF$-radical subgroups is a
conjugation family for $\calF$, and we call this the \emph{Alperin conjugation family}.

Recall that a subgroup $W$ of $P$ is said to be \emph{weakly $\calF$-closed} if for each $\varphi \in \Hom_\calF(W, P)$,
$\varphi(W) = W$. The subgroup $W$ is \emph{strongly $\calF$-closed} if for each subgroup $W'$ of $W$ and each
$\varphi \in \Hom_\calF(W', P)$, $\varphi(W') \leq W$.  We say $W$ is \emph{normal} in $\calF$ 
if $\calF = N_\calF(W)$, and denote by $O_p(\calF)$ the largest such subgroup of $P$.

\section{Bounding the exponent}

The following proposition is slightly misstated in \cite[Proposition \!1.6]{BCGLO2005}, where a normal
$W$ is claimed to be contained in every radical subgroup. For this reason, we
state a correct version here, but the proof in \cite{BCGLO2005} goes through with little
modification. 
\begin{prop}\label{equivnorm}
Let $\calF$ be a fusion system on $P$ and $W \leq P$. The following are equivalent.
\begin{enumerate}
\item[(a)] $W$ is normal in $\calF$.
\item[(b)] $W$ is strongly $\calF$-closed and is contained in every $\calF$-centric, $\calF$-radical 
    subgroup of $P$.
\item[(c)] $W$ is weakly $\calF$-closed and is contained in every subgroup of some conjugation family for $\calF$.
\end{enumerate}
\end{prop}

\begin{lem}\label{goldcent}
    Suppose $P$ has nilpotence class at most $n(p-1) + 1$.
    If $Q$ is a subgroup of $P$ with $C_P(\mho^n(Z(Q))) = Q$, then $Q = P$. 
\end{lem}
\begin{proof}
This is Corollary 6 in \cite{Goldschmidt2008}.
\end{proof}

\begin{prop}\label{poschar}
Let $W$ be a characteristic subfunctor of the center functor such that $W(P)\leq W(Q)$
for all $Q \leq P$ with $C_P(Q) \leq Q$. Then for any fusion system $\calF$ on $P$, either 
there exists a proper $\calF$-centric subgroup $Q$ of $P$ such that $C_{P}(W(Q)) = Q$, or $W(P)$ is normal
in $\calF$. 
\end{prop}
\begin{proof}
Suppose there is no proper $\calF$-centric subgroup $Q$ of $P$ with $C_P(W(Q)) = Q$. We will show that
$W(P)$ is weakly closed in $\calF$. In this case, $W(P) \leq Z(P)$ is contained in every $\calF$-centric
subgroup of $P$, hence in every member of an Alperin conjugation family for $\calF$. 
Thus, by Proposition \ref{equivnorm}, $W(P)$ is in fact normal in $\calF$. 

Let $Q$ be a fully $\calF$-normalized, $\calF$-centric subgroup of $P$. Then by hypothesis, $W(P) \leq W(Q)$.
Let $\alpha \in \Aut_\calF(Q)$. By Alperin's fusion theorem, it suffices to show that $W(P)$ is 
invariant under $\alpha$. We do this by induction on 
$|P:Q|$. If $Q = P$, then $\alpha(W(P)) = W(P)$ since $W(P)$ is a characteristic subgroup of $P$, 
so suppose that $Q < P$.
Then $C_P(W(Q)) > Q$. Let $\beta: W(Q) 
\to R$ be an isomorphism in $\calF$ with $R$ fully $\calF$-normalized. Then by the extension axiom,
$\beta$ extends to a map $\tilde{\beta}: C_P(W(Q)) \to P$. By induction and Alperin's fusion theorem,
we have that $\beta(W(P)) = \tilde{\beta}(W(P)) = W(P)$. But $\beta\alpha|_{W(Q)}$ also extends to $C_P(W(Q))$,
and $\beta\alpha(W(P)) = W(P)$ by the same reasoning. Therefore $\alpha(W(P)) = \inv{\beta}\beta\alpha(W(P))
= W(P)$, and this completes the proof.
\end{proof}

We are now ready to prove Theorem \ref{norm}.
\begin{named}{Theorem \ref{norm}}
Suppose $P$ has nilpotence class at most $n(p-1)+1$ and $\calF$ is a fusion system on $P$. Then $\mho^n(Z(P))$
is normal in $\calF$. 
\end{named}
\begin{proof}
Let $W = \mho^nZ$. If $C_P(Q) \leq Q \leq P$, then $Z(P) \leq Z(Q)$ and so $W(P) = \mho^n(Z(P)) \leq \mho^n(Z(Q))
= W(Q)$. Thus $W$ satisfies the hypotheses of Proposition \ref{poschar}, and Lemma \ref{goldcent}
says that there is no
proper subgroup of $P$ with $C_P(W(Q)) = Q$. Therefore by Proposition \ref{poschar}, $\mho^n(Z(P))$ is 
normal in $\calF$.
\end{proof}

Theorem \ref{main} now follows immediately from Theorem \ref{norm}. The following example generalizes
a remark of Goldschmidt's in \cite{Goldschmidt2008}, and shows that the bound on the exponent of
$Z(P)$ given in Theorem \ref{main} is sharp.

\begin{ex}\label{exsharp}
Let $p$ be an odd prime, let $G = \SL(p+1, q)$ with $|q-1|_p = p^n$, and let $P$ be a Sylow $p$-subgroup of $G$. 
Then $P$ is a isomorphic to $C_{p^n} \wr C_p$. Let $x$ be the wreathing element, a $p$-cycle
permutation matrix. Then $P' = [P, P]$ is isomorphic to a direct product of $p-1$
copies of $C_{p^n}$. Let $P_0 = \langle P', x \rangle$. 
Let $a_1, \dots, a_{p-1}$ be generators for the $p-1$ cyclic groups of $P'$ of order $p^n$. Then $x$ sends
$a_i$ to $a_{i+1}$ for $1 \leq i \leq p-2$ and $a_{p-1}$ to $\inv{a_1}\cdots\inv{a_{p-1}}$.

Let $\mathcal F = \mathcal F_P(G)$. We first claim that
\begin{eqnarray}
\label{opf}
    O_p(\mathcal F) = 1.
\end{eqnarray}
Suppose to the contrary and choose $1 \neq N \leq P$ normal in $\mathcal F$.
Then $N$ contains $\Omega_1(Z(P))$.  Let $Q$ be the unique maximal abelian
subgroup of $P$. Then $\Omega_1(Z(P)) \leq Q$, and the alternating group
$\Alt(p+1) \leq \Aut_{\mathcal F}(Q)$ acts irreducibly on $\Omega_1(Q)$,
so $N$ contains $\Omega_1(Q)$. As $p$ divides $q-1$, the wreathing element $x$ is
diagonalizable, hence $x$ is $\mathcal F$-conjugate to an element in
$\Omega_1(Q)$. It follows that $x \in N$. As $N$ is normal in $P$, we have $[P,
x] \leq N$. Since $[P, x]$ contains an elements of order $p^n$ and $\Alt(p+1)$
also acts irreducibly on the section $Q/\Omega_{n-1}(Q)$, we have that $Q \leq
N$. Now $P = \langle Q, x \rangle$, so $P$ is normal in $\mathcal F$.  But $Q$
is a characteristic subgroup of $P$. This means that $Q$ is normal in $\mathcal
F$, and hence strongly $\mathcal F$-closed. Since $x \nin Q$, and $x$ is
$\mathcal F$-conjugate to an element of $Q$, this is a contradiction. Thus,
(\ref{opf}) holds.

Now as $Z(P)$ has exponent $p^n$, the bound in Theorem \ref{main} is sharp for $\mathcal F$ provided the class 
of $P$ is $n(p-1) + 1$. For this it suffices
to show that $P_0$ has class $n(p-1)$, that is, $P_0$ is of maximal class.

For $n = 1$, $P_0$ is of maximal class $p-1$. We show by induction on $n$ that
\begin{eqnarray}
\label{exeqn}
[P', x; p-1] = \Omega_{n-1}(P')
\end{eqnarray}
and this will complete the proof.
Factoring by $\Omega_{n-1}(P')$, we have $[P'/\Omega_{n-1}(P'), x; p-1] = 1$ so that $[P', x; p-1]
\leq \Omega_{n-1}(P')$ in any case.

Suppose first that $n=2$. By direct computation,

\[
[a_1, x; p-1] = \prod_{k=0}^{p-2} a_{k+1}^{(-1)^k \binom{p-1}{k} - 1}.
\]
The sum of the exponents of the $a_i$ in $[a_1, x; p-1]$ is

\[
-p + 1 + \sum_{k=0}^{p-2} (-1)^k \binom{p-1}{k} = -p + 1 + (1-1)^{p-1} - \binom{p-1}{p-1} = -p.
\]
This means that $[a_1, x; p-1]$ lies outside the sum-zero submodule (which is the unique maximal
submodule) for the action of $x$ on
$\Omega_1(P')$, and so $[P', x; p-1] = \Omega_1(P')$. 

Let $n \geq 3$ be arbitrary. Let $N = [P', x; p-1]$. By induction we have that $N$ contains 
$[\Omega_{n-1}(P'),x;p-1]=\Omega_{n-2}(P')$ and by the $n=2$ case, 
we know that $N$ covers $\Omega_1(P'/\Omega_{n-2}(P'))$ modulo $\Omega_{n-2}(P')$. Therefore,
$N = \Omega_{n-1}(P')$, proving (\ref{exeqn}).

It now follows that $P$ has class $n(p-1) + 1$ while $Z(P)$ has exponent $p^n$, and so the bound of Theorem
\ref{main} is sharp. 
\end{ex}

\section{A factorization theorem}

We now turn to the proof of Theorem \ref{fact}. We will need 
a version of the Frattini argument due to Onofrei and Stancu \cite[Proposition \!3.7]{OnofreiStancu2009}.

\begin{prop}\label{frattini}
Let $\calF$ be a fusion system on $P$ and suppose $Q \leq P$ is normal in $\calF$. Then
\[
\calF = \langle \,PC_\calF(Q),\,\, N_\calF(QC_P(Q)) \,\rangle.
\]
\end{prop}

\begin{lem}\label{goldj}
Suppose $P$ is a $p$-group, $Q \norm P$, and $C_P(\mho^1(Z(Q))) = Q$. Then $J(P) \leq Q$.
\end{lem}
\begin{proof}
This is Lemma $8$ in \cite{Goldschmidt2008}.
\end{proof}

The \emph{Thompson ordering} on subgroups of $P$ is defined by
\[
Q \leq_P Q' \quad \text{iff} \quad |N_P(Q)| \leq |N_P(Q')|\quad \text{or} \quad|N_P(Q)| = |N_P(Q')|\quad \text{and}\quad |Q| \leq |Q'|.
\]
We are now ready to prove

\begin{named}{Theorem \ref{fact}}
Let $\calF$ be a fusion system on $P$. Then
\[
\calF = \langle\,\,C_\calF(\mho^1(Z(P)),\, N_{\calF}(J(P))\,\,\rangle.
\]
\end{named}

\begin{proof}
Write $\calF' = \langle\, C_\calF(\mho^1(Z(P))),\,\, N_\calF(J(P)) \,\rangle$.
Since each $\calF$-centric subgroup of $P$ contains $Z(P)$, it suffices by Alperin's fusion theorem to
prove that $N_\calF(Q) \cin \calF'$ for all $Q \leq P$ with $Z(P) \leq Q$. We do this by induction on
the Thompson ordering. If $Q = P$, then $N_\calF(Q) \cin N_\calF(J(P)) \cin \calF'$, since $J(P)$ is
a characteristic subgroup of $P$,
so suppose that $Q <_P P$ with $Z(P) \leq Q$ and that $N_\calF(Q') \cin \calF'$ for all $Q' >_P Q$
with $Z(P) \leq Q'$. 

First we reduce to the case where $Q$ is fully $\calF$-normalized.
Suppose $Q$ is not fully $\calF$-normalized. By \cite[Lemma \!2.2]{KessarLinckelmann2008}, there exists $\alpha: N_P(Q) \to P$
such that $\alpha(Q)$ is fully $\calF$-normalized. Note that $\alpha(Q) >_P Q$, 
and since $R >_P Q$ for every $R \leq P$ with $|N_P(Q)| \leq |R|$, we have by induction
and Alperin's fusion theorem that $\alpha$ is in $\calF'$. Also note that $\alpha(N_P(Q)) \leq N_P(\alpha(Q))$;
we still denote by $\alpha$ the induced morphism $N_P(Q) \to N_P(\alpha(Q))$. 
Let $\varphi: R_1 \to R_2$ be a morphism in $N_\calF(Q)$, and let $\tilde{\varphi}$ be an
extension to $QR_1 \leq N_P(Q)$. Then $\alpha\tilde{\varphi}\inv{\alpha}: \alpha(Q)\alpha(R_1) \to \alpha(Q)\alpha(R_2)$ restricts 
to an automorphism of $\alpha(Q)$, whence is contained in $\calF'$ by induction. But $\alpha$ is
in $\calF'$, so $\varphi$ is in $\calF'$ too. Thus $N_\calF(Q) \cin \calF'$, so henceforth we assume $Q$ 
is fully $\calF$-normalized.

For brevity, set $W = \mho^1(Z(Q))$, $N = N_P(Q)$, and $C = C_N(W)$. Then $C \norm N$, so that $N_P(C) \geq N$.
Suppose first that $C = Q$. Then by Lemma \ref{goldj}, we have $J(N) \leq Q$. As $J(N) \norm N_P(N)$, either $J(N) >_P Q$ or $N = P$.
In the first case, since $Z(P) \leq J(N)$ and $J(N) = J(Q)$ is a characteristic subgroup of $Q$, 
we apply induction to get $N_\calF(Q) \cin N_\calF(J(N)) \cin \calF'$.
In the second case we have $J(P) \leq Q$, so $J(P) = J(Q)$, and hence $N_\calF(Q) \cin N_\calF(J(P)) \cin \calF'$ here as well. 

Assume now that $C > Q$. Then $C >_P Q$ because $C \norm N$. Looking to see that $W \norm N_\calF(Q)$, we apply
Proposition \ref{frattini} in this normalizer to get
\[
N_\calF(Q) = \langle\,\, NC_{N_\calF(Q)}(W),\, N_{N_\calF(Q)}(C)\,\,\rangle.
\]
Since $C$ contains $Z(P)$, we have by induction that $N_{N_\calF(Q)}(C) \cin N_\calF(C) \cin \calF'$, so to complete the proof, it suffices to 
show that $NC_{N_\calF(Q)}(W) \cin C_{\calF}(\mho^1(Z(P)))$. To see this, 
let $R_1, R_2 \leq N$, and let $\varphi: R_1 \to R_2$ be a morphism in $NC_{N_\calF(Q)}(W)$. Then there exists
$x \in N$ such that $\varphi$ extends to an $\calF$-map $\tilde{\varphi}: WR_1 \to WR_2$ with $\tilde{\varphi}\,|_W = c_x$,
the conjugation map induced by $x$.
But since $Q$ contains $Z(P)$, it follows that $W = \mho^1(Z(Q)) \geq \mho^1(Z(P))$, and so 
$\tilde{\varphi}\,|_{\mho^1(Z(P))} = c_x\,|_{\mho^1(Z(P))} = \id_{\mho^1(Z(P))}$. Therefore, $\varphi \in
C_{\calF}(\mho^1(Z(P)))$, as was to be shown. We conclude that $N_{\calF}(Q) \cin \calF'$ and the result follows.
\end{proof}

\begin{remark}
In \cite[Theorem \!4.1]{DiazGlesserMazzaPark2009}, the authors prove in part that for any fusion system $\calF$ on $P$, 
$\mho^1(Z(P)) \cap Z(N_\calF(J(P))) \leq Z(\calF)$ by reducing to the group case. Theorem \ref{fact} gives
a reduction-free proof of this fact.
\end{remark}

\section{Acknowledgements}
We would like to thank Ron Solomon for encouraging us to take up this work.

\bibliographystyle{amsplain}
\bibliography{/home/jlynd/math/research/mybib}

\end{document}